\documentclass[journal,10pt,twocolumn]{IEEEtran}
\usepackage{enumerate}
\usepackage{cite}
\usepackage{amsmath}
\usepackage{amscd}
\usepackage{amssymb}
\usepackage{latexsym}
\usepackage{array}
\usepackage{tabularx}
\usepackage[ruled]{algorithm}
\usepackage{algorithmic}
\usepackage{epsfig,subfigure}
\usepackage{changepage}
\usepackage{stfloats}
\usepackage{multirow}
\usepackage{multirow}
\usepackage{bigstrut}
\usepackage{array}
\usepackage{epstopdf}
\usepackage{graphicx}

\newcommand{\PreserveBackslash}[1]{\let\temp=\\#1\let\\=\temp}
\newcolumntype{C}[1]{>{\PreserveBackslash\centering}p{#1}}
\newcolumntype{R}[1]{>{\PreserveBackslash\raggedleft}p{#1}}
\newcolumntype{L}[1]{>{\PreserveBackslash\raggedright}p{#1}}

\newtheorem{prop}{Proposition}

\begin{document}
%
\title{Alternating Strategies Are Good For \\ Low-Rank Matrix Reconstruction}


\author{Kezhi~Li, 
        Martin Sundin, 
        Cristian~R.~Rojas, 
        Saikat Chatterjee, 
        Magnus Jansson
\thanks{K. Li, M. Sundin, C. R. Rojas, S. Chatterjee and M. Jansson are with the with the
ACCESS Linnaeus Center, Electrical Engineering, KTH – Royal
Institute of Technology, S-100 44 Stockholm, Sweden. (e-mail: kezhi@kth.se, masundi@kth.se, crro@kth.se, sach@kth.se, janssonm@kth.se). This work was partially supported by the Swedish Research Council under contract 621-2011-5847 and by the ACCESS Linnaeus Centre. }
}

%

\maketitle

\maketitle
\begin{abstract}
This article focuses on the problem of reconstructing low-rank matrices from underdetermined measurements using alternating optimization strategies. We endeavour to combine an alternating least-squares based estimation strategy with ideas from the alternating direction method of multipliers (ADMM) to recover structured low-rank matrices, such as Hankel structure. We show that merging these two alternating strategies leads to a better performance than the existing alternating least squares (ALS) strategy. The performance is evaluated via numerical simulations.


\end{abstract}


\section{Introduction}

The low-rank matrix reconstruction problem arises naturally in many fields, such as system identification \cite{Fazel-Hankel,Jansson-Subspace}, spectral imaging \cite{Signoretto-Tensor} and audio signal processing \cite{Arberet-Reverberant}. Suppose an $r$-rank matrix $\mathbf{X}$ has size $n_1 \times n_2$, $r \ll \min(n_1, n_2)$, the objective is to recover $\mathbf{X}$ from the measurement equation
\begin{equation}\label{eq:y=AX+e}
\mathbf{y} = \mathcal{A}(\mathbf{X})+ \mathbf{e},
\end{equation}
where $\mathbf{y}\in \mathbb{R}^{m\times 1}$ is the measurement vector, $\mathcal{A}$ denotes a known sensing function $\mathbb{R}^{n_1n_2} \rightarrow \mathbb{R}^{m \times 1}$, and $\mathbf{e}$ is assumed to be zero-mean noise with given covariance $E(\mathbf{e}\mathbf{e}^T)=\mathbf{C} \in \mathbb{R}^{m \times m}$. Usually $m<n_1\times n_2$ thus the number of coefficients of $\mathbf{X}$ is larger than the number of measurements, so (\ref{eq:y=AX+e}) is underdetermined.
Specifically we consider the case where $\mathcal{A}$ is a linear operator so that (\ref{eq:y=AX+e}) can be rewritten equivalently as the product of an operator $\mathbf{A}$ and a low-rank $\mathbf{X}$
\begin{equation}
\mathcal{A}(\mathbf{X}) =  \mathbf{A} \text{vec}(\mathbf{X}),
\end{equation}
where $\mathbf{A}\in \mathbb{R}^{m \times n_1n_2}$. 


Compared to nuclear norm minimization \cite{RECHT-Guaranteed}, an \emph{alternating technique} solution provides faster computation, high accuracy and hence is useful for solving such underdetermined problems based on different criteria such as the maximum likelihood (ML) estimator \cite{Tang-LowerBounds} or the least squares (LS) estimator \cite{Zachariah-ALSforLR}. Typically alternating approaches provide locally optimal solutions through iterations. Each iteration leads to the best solution of a set of variables given another set of known variables found in the previous iteration.
Here our hypothesis is that if the \emph{updating directions} of previous iterations are also considered in each iteration, the reconstruction will be improved and accelerated because the potential feasible set of solutions is narrowed in each iteration.
In this regard, few relevant attempts have been made in matrix completion \cite{Wen-SolvingA}, or to update solutions using a gradient descent method \cite{Becker-randomizedLow, Lee-PracticalLarge}.

 In this paper we develop two algorithms and both of them are based on the alternating technique. First we simply modify the conventional alternating least squares (ALS) algorithm proposed in \cite{Zachariah-ALSforLR} to derive a new algorithm called alternating linear estimator (ALE) for low rank matrices with linear structure. Then based on the ALS and ALE, we develop a novel algorithm called alternating direction least squares (ADLS) that endeavours to validate our hypothesis on updating directions by fusing two alternating strategies. It utilizes the alternating strategy with the help of an updating direction for structured matrix reconstruction. Inspired by the ALS, the proposed framework is based on running the LS estimation to update the low rank component matrices $\mathbf{L}$, $\mathbf{R}$ and $\mathbf{X}$ iteratively, where $\mathbf{LR}= \mathbf{X}$. In our new approach, to incorporate direction update knowledge, the new $\mathbf{L},\mathbf{R}$ are calculated by solving optimization problems involving the augmented Lagrangian. This method is able to push variables converging to solutions more efficiently, as shown in the standard alternating direction method of multipliers (ADMM) \cite{Boyd-DistributedOpt, Wen-SolvingA}. The new algorithm also inherits the capability of ALS of handling structured matrices, such as Hankel structure. 
 The simulation results are compared with the classical ALS performance and Cram\'{e}r-Rao bounds (CRBs) to show the effectiveness of the proposed approach.


{\emph{Notations}}: Bold letters are used to denote a vector or a matrix. For vectors, $||\cdot||_1, ||\cdot||_2, ||\cdot||_{\infty}$ represent the $l_1, l_2$ and $l_{\infty}$ norms, respectively. For matrices, $\mathbf{A}^T$ and $\mathbf{A}^{\dagger}$ denote the transpose and Moore-Penrose pseudoinverse of $\mathbf{A}$. Moreover, $||\cdot||_F$ represents the Frobenius norm and $||\mathbf{x}||_\mathbf{W} \triangleq \sqrt{\mathbf{x}^T\mathbf{W}\mathbf{x}}$. $\chi_r \ \triangleq \{ \mathbf{A} \in \mathbb{R}^{n_1\times n_2}: \text{rank}(\mathbf{A})=r \}$ denotes the set of rank $r$ matrices. $\text{vec}(\mathbf{A})$ represents the column vector of concatenated $\mathbf{A}$'s columns, and ($\text{mat}_{n_1,n_2}$) is ($\text{vec}$)'s inverse operation to convert a vector to a matrix of size $n_1 \times n_2$. $\bigtriangledown_{\mathbf{X}}\{f\}$ denotes the partial derivative operation of the function $f$ on $\mathbf{X}$. p.s.d. is the short form for positive semidefinite.


\section{ALS for Low-Rank Matrix Reconstruction}
The alternating least-squares approach was developed in \cite{Zachariah-ALSforLR}. For an $r$-rank matrix $\mathbf{X}$ satisfying (\ref{eq:y=AX+e}) with noise covariance $\mathbf{C}$, the weighted least-squares estimator is
\begin{equation}\label{eq:main_X}
\hat{\mathbf{X}} = \arg\min_{ \footnotesize{\mathbf{X} \in \chi_r }}{ ||\mathbf{y}- \mathcal{A}(\mathbf{X})||^2_{\mathbf{C}^{-1}}}.
\end{equation}
To rewrite (\ref{eq:main_X}) in terms of the standard $2$-norm, the measurements and sensing operator can be prewhitened by forming $\bar{\mathbf{y}}= \mathbf{C}^{-1/2}\mathbf{y}$ and $\bar{\mathbf{A}} = \mathbf{C}^{-1/2} \mathbf{A}$. Expressing $\mathbf{X} = \mathbf{LR}$ where $\mathbf{L} \in \mathbb{R}^{n_1 \times r}$ and $\mathbf{R} \in \mathbb{R}^{r \times n_2}$, the square of residuals becomes
\begin{equation}\label{eq:J(LR)}
\begin{split}
\mathbf{J}(\mathbf{L},\mathbf{R}) &\triangleq  || \bar{\mathbf{y}} - \bar{\mathcal{A}}(\mathbf{LR}) ||^2_2 \\
&= || \bar{\mathbf{y}}- \bar{\mathbf{A}}(\mathbf{I}_{n_1} \otimes \mathbf{L}) \text{vec}(\mathbf{R}) ||_2^2 \\
&= || \bar{\mathbf{y}}- \bar{\mathbf{A}}(\mathbf{R}^T \otimes \mathbf{I}_{n_2}) \text{vec}(\mathbf{L}) ||_2^2.
\end{split}
\end{equation}
The cost function $\mathbf{J}(\mathbf{L},\mathbf{R})$ is minimized cyclically by
\begin{equation}
\begin{split}
\hat{\mathbf{R}}&= \arg\min_{\mathbf{R}} || \bar{\mathbf{y}}- \bar{\mathbf{A}}(\mathbf{I}_{n_1} \otimes \hat{\mathbf{L}}) \text{vec}(\mathbf{R}) ||_2^2, \\
\hat{\mathbf{L}}&= \arg\min_{\mathbf{L}} || \bar{\mathbf{y}}- \bar{\mathbf{A}}(\hat{\mathbf{R}}^T \otimes \mathbf{I}_{n_2}) \text{vec}(\mathbf{L}) ||_2^2.
\end{split}
\end{equation}
The iterations of $\hat{\mathbf{R}}$ and $\hat{\mathbf{L}}$ continue until the residual $||\bar{\mathbf{y}}- \bar{\mathcal{A}}(\mathbf{LR})||_2^2$ no longer decreases. Specifically, we calculate the analytical solution $\text{vec}(\hat{\mathbf{R}}) = [\bar{\mathbf{A}} (\mathbf{I}_{n_1} \otimes \mathbf{L}) ]^{\dagger} \bar{\mathbf{y}}$ given $\mathbf{L}$ and $\text{vec}(\hat{\mathbf{L}}) = [\bar{\mathbf{A}} (\mathbf{R}^T \otimes \mathbf{I}_{n_2}) ]^{\dagger} \bar{\mathbf{y}}$ given $\mathbf{R}$. ALS is also capable of recovering structured low-rank matrices such as Hankel, Toeplitz, as well as p.s.d. matrices. In this case a projection step $\hat{\mathbf{X}} \triangleq \mathcal{P}(\mathbf{L}\hat{\mathbf{R}})$ is added after updating $\hat{\mathbf{R}}$ using a ``lift and project" approach and a new $\bar{\mathbf{R}}$ is calculated by the least-squares estimation
\begin{equation}
\bar{\mathbf{R}}=\min_{\mathbf{R}}|| \mathbf{LR}- \hat{\mathbf{X}} ||_F^2.
\end{equation}
$\bar{\mathbf{L}}$ can be updated likewise. ALS has been verified effectively for recovering low-rank matrices of large sizes.


\section{Alternating Linear Estimator}

In this section we develop a simple but efficient algorithm for low rank matrices with linear structure, called the alternating linear estimator (ALE). We assume that the low rank matrix has linear structure which means that $\mathbf{X} \in \mathbb{R}^{n_1\times n_2}$ can be decomposed as
\begin{equation}
\mathbf{X} = \mathcal{S}_\chi (\mathbf{h})
\end{equation}
where $\mathbf{h} \in \mathbb{R}^p$ is a parametrization of $\mathbf{X}$ and $\mathcal{S}_\chi : \mathbb{R}^p \to \mathbb{R}^{n_1 \times n_2}$ is a linear map parametrizing the linear structure of $\mathbf{X}$. Taking  $\mathbf{X}$ as a Hankel matrix for instance, $\mathbf{h}$ can contain the first column and last row of $\mathbf{X}$. We denote the pseudoinverse of $\mathcal{S}_\chi$ by $\mathcal{T}_\chi$, i.e.
\begin{equation}
\mathcal{T}_\chi (\mathbf{X}) = \arg \min_\mathbf{h} || \mathbf{X} - \mathcal{S}_\chi (\mathbf{h})||_F^2.
\end{equation}
The idea of ALE is to iteratively update $\mathbf{L},\mathbf{R}$ and apply the lift-and-project (or composite mapping) method to project the matrix to its linear structure. As an initial least squares estimate we set
\begin{equation}\label{eq:h_0}
\mathbf{h}_0 = \arg \min_\mathbf{h} || \bar{\mathbf{y}} - \bar{\mathcal{A}}(\mathcal{S}_\chi ( \mathbf{h} )) ||_2^2.
\end{equation}
We then apply lift-and-project to the estimate $\mathbf{h}_0$. The details of the algorithm are summarized in Algorithm 1. \\
\begin{adjustwidth}{-0.1cm}{}
\begin{tabular}{l}
\hline
\textbf{Algorithm 1}: Alternating Linear Estimator \\
\hline \vspace{-3mm}\\
\textbf{Input}: sensing operator $\bar{\mathbf{A}}$, measurements $\bar{\mathbf{y}}$, rank $r$. \\
\textbf{Set}: residual bound $\epsilon$, max number of iterations $k_\text{max}$. \\
\textbf{Initialize}: $\mathcal{S}_\chi$, $\mathcal{T}_\chi$, $\mathbf{h}_0$\\
\textbf{Interations}: \ \ $k=0, 1, \cdots, k_{\text{max}}$\\
\ \ $\mathbf{L}_{k+1} = \left\{ \text{first $r$ columns of } \mathcal{S}_\chi (\mathbf{h}_k) \right\}$\\
\ \ (Or $\mathbf{L}_{k+1} = \arg \min_\mathbf{L} || \mathcal{S}_\chi (\mathbf{{h}}_k) - \mathbf{L} \mathbf{R}_k||_F^2 =\mathcal{S}_\chi (\mathbf{{h}}_k) \mathbf{R}_k^\dagger$ )\\
\ \ $\mathbf{R}_{k+1} = \arg \min_\mathbf{R} || \mathcal{S}_\chi (\mathbf{{h}}_k) - \mathbf{L}_{k+1} \mathbf{R}||_F^2 =\mathbf{L}_{k+1}^\dagger \mathcal{S}_\chi (\mathbf{{h}}_k)$\\
\ \ $\mathbf{h}_{k+1} = \mathcal{T}_{\chi} (\mathbf{L}_{k+1}\mathbf{R}_{k+1})$\\
\ \   $\backslash \backslash$ stop criterion  \\
\ \ If $||\mathcal{S}_\chi (\mathbf{{h}}_{k+1}) - \mathbf{L}_{k+1} \mathbf{R}_{k+1}||_\infty \leq \epsilon$, break; \\
\textbf{Output}: $\mathbf{X} = \mathbf{L}_{k+1}\mathbf{R}_{k+1}$. \\
\hline\\
\end{tabular}
\end{adjustwidth}

Since the algorithm only uses projections and mappings which can be performed iteratively, the algorithm is scalable and thus appropriate for large scale problems. 
Simulations show that ALE has a better performance better than ALS; see details in Section \ref{sec:Simulations}.

\section{Alternating Direction Least-Squares Estimator}
\subsection{ADMM Embedded to Alternating Iteration}
The ALS estimator leverages the low-rank constraints by forming $\mathbf{X}=\mathbf{LR}$ and updates the factors iteratively. However, in ALS the changing directions of $\mathbf{L},\mathbf{R}$ are not recorded. On the other hand, Algorithm 1 complies with (\ref{eq:h_0}) only at the initial step, yet in the iterations it satisfies the low rank and structured constraints alternatively. So the result might drift from the true solution since it does not impose equation (\ref{eq:y=AX+e}) directly in iterations. To exploit the direction information and overcome the drifting problem, we adopt the alternating direction augmented Lagrangian methods to strengthen the ALS algorithm \cite{Boyd-DistributedOpt}. The alternating direction method of multipliers (ADMM) is embedded in each iteration.
ADMM is widely used in various applications to solve convex optimization problems effectively and can be used to construct distributed optimization algorithms \cite{Boyd-DistributedOpt,Wahlberg-AnADMM}. To minimize $f(\mathbf{R})=|| \bar{\mathbf{y}}- \bar{\mathbf{A}}(\mathbf{I}_{n_1} \otimes \mathbf{L}) \text{vec}(\mathbf{R}) ||_2^2$ and $g(\mathbf{Z})=\mu ||\mathbf{LZ}-\hat{\mathbf{X}}||^2_{F}$ where $\hat{\mathbf{X}}$ denotes the projection result on the structure $\chi$ (like Hankel) and $\mu$ is a parameter to balance the weights of $f$ and $g$, the problem can be rewritten as
\begin{equation}\label{eq:ADMM_f_I}
\begin{split}
\text{minimize}\ f(\mathbf{R})+ g(\mathbf{Z})  \\
\text{subject to}\ \mathbf{R}=\mathbf{Z},
\end{split}
\end{equation}
where $f$ and $g$ pursue solutions within measurement and low-rank constraints, respectively. To robustify the algorithm, an augmented Lagrangian is introduced. Combining the linear and quadratic terms, it becomes
\begin{equation}
L_{\lambda}(\mathbf{R},\mathbf{Z},\mathbf{U})= f(\mathbf{R})+g(\mathbf{Z}) + \lambda||\mathbf{R}-\mathbf{Z}+\mathbf{U}||^2_2,
\end{equation}
where $\lambda>0$ is a penalty parameter and $\mathbf{U}$ is a scaled dual variable associated with the constraint $\mathbf{R}=\mathbf{Z}$. The iteration of updating the left matrix can be formulated by minimizing the augmented Lagrangian over $\mathbf{R}$ and $\mathbf{Z}$ cyclically:
\begin{eqnarray}\label{eq:ADMM_XZU}
\mathbf{R}_{k+1} &= \arg\min_\mathbf{R} L_{\lambda}(\mathbf{R},\mathbf{Z}_k,\mathbf{U}_k) \ \ \ \ \  // \ \mathbf{R}\text{-min}\label{eq:ADMM_R}\\
\mathbf{Z}_{k+1} &= \arg\min_\mathbf{Z} L_{\lambda}(\mathbf{R}_{k+1},\mathbf{Z}_k,\mathbf{U}_k) \  // \ \mathbf{Z}\text{-min} \label{eq:ADMM_Z}\\
\mathbf{U}_{k+1} &= \mathbf{U}_{k} + \lambda(\mathbf{R}_{k+1} - \mathbf{Z}_{k+1}) \ \label{eq:ADMM_U}// \ \text{dual-update}
\end{eqnarray}
Firstly $\mathbf{Z}_k$ and $\mathbf{U}_k$ are fixed and we minimize the augmented Lagrangian over $\mathbf{R}$. The second step is to pursue a solution to $\mathbf{Z}$ under the constraints on $\mathbf{X}$. At last the dual variable $\mathbf{U}_k$ is updated in the third step. The algorithm runs iteratively until primal and dual residuals become smaller than the terminating bounds $||\mathbf{R}_k- \mathbf{Z}_k||_2 \leq \epsilon^{\text{pri}}, ||\lambda (\mathbf{Z}_k- \mathbf{Z}_{k-1})||_2 \leq \epsilon^{\text{dual}}$. $\mathbf{L}$ can be updated in a similar manner. Though $f,g$ are simple quadratic functions, the ADMM approach benefits from the use of updating directions of previous iterations, and it is capable of pursuing solutions satisfying constraints of measurements, low-rank and structured \cite{Boyd-DistributedOpt,Wahlberg-AnADMM} at the same time.

\subsection{Alternating Direction Least-Squares}
The ADLS algorithm combines the features of the ALE and the ALS algorithm, and embeds ADMM loops to update $\mathbf{L},\mathbf{R}$, respectively.
Because functions $f$ and $g$ are convex and quadratic, the augmented Lagrangian over $\mathbf{R},\mathbf{Z}, \mathbf{U}$ can be minimized explicitly (see Appendix). In particular $\mathbf{R}$ in (\ref{eq:ADMM_R}) becomes
\begin{equation}\label{eq:update_R}
\text{vec}(\mathbf{R}_{k+1}) = \left( \mathbf{P}_{k}^T{\mathbf{P}_{k}} + \lambda \mathbf{I} \right)^{-1} \left( \mathbf{P}_{k}^T \bar{\mathbf{y}} +\lambda(\mathbf{z}_{k}-\mathbf{u}_k) \right),
\end{equation}
where the lowercase letters represent the vectorized matrices, e.g. $\mathbf{z}=\text{vec}(\mathbf{Z})$; $\mathbf{P}_{k} = \bar{\mathbf{A}}(\mathbf{I}_{n_1} \otimes \mathbf{S}_{k})$ and $k$ denotes the index of iterations; $\mathbf{S}$ is the dual of $\mathbf{L}$ and $\mathbf{S}_0=\mathbf{L}_0$. For least squares $g(\mathbf{Z})$ the $\mathbf{Z}\text{-min}$ step in (\ref{eq:ADMM_Z}) is equal to
 \begin{equation}\label{eq:update_Z}
\mathbf{Z}_{k+1}=\left( \mu \mathbf{S}_{k}^T{\mathbf{S}_{k}} + \lambda \mathbf{I} \right)^{-1} \left( \mu \mathbf{S}_{k}^T \hat{\mathbf{X}}_{k} +\lambda(\mathbf{R}_{k}+\mathbf{U}_k) \right).
\end{equation}
In the process of using ADMM to update the left matrix $\mathbf{L}$:
\begin{equation}\label{eq:update_L}
\text{vec}(\mathbf{L}_{k+1}) = \left( \mathbf{Q}_{k}^T{\mathbf{Q}_{k}} + \lambda \mathbf{I} \right)^{-1} \left( \mathbf{Q}_{k}^T \bar{\mathbf{y}} +\lambda(\mathbf{S}_{k}-\mathbf{T}_k) \right),
\end{equation}
where $\mathbf{Q}_{k} = \bar{\mathbf{A}}( \mathbf{Z}_{k+1}^T \otimes \mathbf{I}_{n_2})$, and
 \begin{equation}\label{eq:update_S}
\mathbf{S}_{k+1} = \left(\mu \hat{\mathbf{X}}_k\mathbf{Z}_{k+1}^T+\lambda(\mathbf{L}_{k+1}+\mathbf{T}_k)\right)\left(\mu \mathbf{Z}_{k+1}\mathbf{Z}_{k+1}^T+\lambda \mathbf{I}\right)^{-1}.
\end{equation}
Similar to ALE,  the structured $\hat{\mathbf{X}}$ is obtained in advance by projecting $\mathbf{SZ}$ to the constraints set $\chi$,
\begin{equation}\label{eq:update_X}
\hat{\mathbf{X}}_{k+1} \triangleq \mathcal{P}_\chi(\mathbf{S}_{k+1}\mathbf{Z}_{k+1}),
\end{equation}
where $\mathcal{P}_\chi(.)= \mathcal{S}_\chi(\mathcal{T}_\chi(.))$. Assuming that the low rank $\mathbf{X}$ belongs to the set of Hankel matrices $\chi$, the pseudo-code of our algorithm ADLS is presented in Algorithm 2. \\
\begin{adjustwidth}{-0.1cm}{}
\begin{tabular}{l}
\hline
\textbf{Algorithm 2}: Alternating Direction Least-Squares \\
\hline \vspace{-3mm}\\
\textbf{Input}: sensing operator $\bar{\mathbf{A}}$, measurements $\bar{\mathbf{y}}$, rank $r$. \\
\textbf{Set}: residual bound $\epsilon$, weight $\mu$, tunable parameters $\lambda, \lambda'$, \\max number of iterations $k_\text{max}$. \\
\textbf{Initialize}: Perform one iteration of Algorithm 1 to obtain\\
 $\mathbf{h}_0,\mathbf{L}_0, \mathbf{R}_0$, and $\hat{\mathbf{X}}_1=\mathbf{L}_0  \mathbf{R}_0$. $\mathbf{S}_{0}=\mathbf{L}_{0}$, $\mathbf{U}_{0}=\mathbf{0}$, $\mathbf{Z}_{0}=\mathbf{R}_{0}$,\\
  $\mathbf{T}_{0}=\mathbf{0}$.\\
\textbf{Iteration}: for $k = 0,1, \cdots$, $k_\text{max}$ \ \ do \\
$\backslash \backslash$ inner loop to calculate the right matrix  \\
\ \ $\mathbf{R}_{k+1}$ is updated by (\ref{eq:update_R}); \\
\ \ $\mathbf{Z}_{k+1}$  is updated by (\ref{eq:update_Z}); \\
\ \ $\mathbf{U}_{k+1} = \mathbf{U}_{k}+ \lambda'(\mathbf{R}_{k+1} - \mathbf{Z}_{k+1})$;\\
$\backslash \backslash$ inner loop to calculate the left matrix  \\
\ \ $\mathbf{L}_{k+1}$ is updated by (\ref{eq:update_L}) ; \\
\ \ $\mathbf{S}_{k+1}$  is updated by (\ref{eq:update_S}); \\
\ \ $\mathbf{T}_{k+1} = \mathbf{T}_{k}+ \lambda'(\mathbf{L}_{k+1} - \mathbf{S}_{k+1})$;\\
$\backslash \backslash$ update $\hat{\mathbf{X}}$ with constraints  \\
\ \ $\mathbf{h}_{k+1} = \mathcal{T}_{\chi} (\mathbf{S}_{k+1}\mathbf{Z}_{k+1})$; \\
\ \ $\hat{\mathbf{X}}_{k+1}= \mathcal{S}_{\chi}(\mathbf{h}_{k+1})$; \\
$\backslash \backslash$ stop criterion  \\
\ If $||\mathcal{S}_\chi (\mathbf{{h}}_{k+1}) - \mathbf{S}_{k+1} \mathbf{Z}_{k+1}||_\infty \leq \epsilon$, break; \\
\textbf{Output}: $\mathbf{X} = \mathbf{S}_{k+1}\mathbf{Z}_{k+1}$. \\
\hline\\
\end{tabular}
\end{adjustwidth}

Algorithm 2 follows the steps to calculate $\mathbf{R},\mathbf{L}$ and $\hat{\mathbf{X}}$ alternately. So it also belongs to the alternating strategies. In contrast to ALS or ALE, in ADLS each step is derived from the analytical solution to its augmented Lagrangian. Compared to ALS, Algorithm 2 uses the projection only once in each iteration, and it is able to achieve a balance between the measurement constraints by function $f$ and the $\chi$ constraints by function $g$ by adjusting $\mu$. Meanwhile compared to ALE, ADLS also incorporates function $f$ in the updates so that it prevents the scenario that the output of $\mathbf{SZ}$ drifts away from the true solution to (\ref{eq:main_X}).

 \subsection{Algorithm Analysis}

 In ADLS, $\mathbf{L}$ and $\mathbf{R}$ are updated within an inner loop subject to the constraints rather than calculated directly in ALS. So the algorithm is strengthened in efficiency and robustness. When the constraint set $\chi$ denotes the low rank and Hankel set, ADLS can be initialized by the parametrization $\mathbf{h}_0$ as shown in Algorithm 2, as with any other structure which can be decomposed linearly. Otherwise ADLS is initialized by the singular value decomposition (SVD): $\mathbf{U}  \mathbf{\Sigma}  {\mathbf{V}}^T = \text{mat}_{n_1,n_2}(\bar{\mathbf{A}}^T \bar{\mathbf{y}})$, where $\mathbf{L}_{0}=\mathbf{U} \sqrt{\mathbf{\Sigma }_r}$ and $\mathbf{R}_0= \sqrt{\mathbf{\Sigma }_r} {\mathbf{V}}^T$ are rescaled with the square root of singular values to balance the norms of sub-matrices, and $\mathbf{\Sigma}_r$ is the $\mathbf{\Sigma}$ truncated to the $r$th singular value. Since the SVD is used only once for the initialization, the proposed algorithm is computationally superior to other algorithms that are based on SVD.

 For convex problems it was proven that the residuals and the cost function converge to zero and an optimal value respectively as ADMM proceeds under mild assumptions, and the rate of convergence is determined by the choices of $\lambda, \lambda'$ \cite{Boyd-DistributedOpt}. In ADLS, $\lambda, \lambda'$ need to be tuned carefully. Ideas from \cite{Boyd-DistributedOpt} may be adopted as guidances to tune parameters. In practice, we set $\lambda, \lambda'\in [0,1]$ based on numerical experiments. ADLS can also be extended to some special cases of nonconvex $\chi$ such as cardinality or boolean constraints, in which $\mathcal{P}_{\chi}$ in (\ref{eq:update_X}) should be changed accordingly.

\section{Numerical Experiments}

\subsection{Performance Measure and Cram\'{e}r-Rao Bounds}

Extensive simulations have been performed using MATLAB to test the recovery performances of unstructured or structured low-rank matrices contaminated with noise. 
Similar to settings in \cite{Zachariah-ALSforLR}, a Hankel matrix $\mathbf{X}$ is generated randomly by creating a matrix with elements from an i.i.d. $\mathcal{N}(0,1)$ and fitting $\mathbf{h}$ using Prony's method \cite{Hayes-StatisticalDig}. We use the signal-to-reconstruction error ratio (SRER), or namely the inverse of the normalized mean square error (NMSE) to measure the reconstruction result
\begin{equation}
    \text{SRER}= \frac{1}{\text{NMSE}} \triangleq \frac{E\left[ ||\mathbf{X}||_F^2 \right]}{E\left[ ||\mathbf{X}-\hat{\mathbf{X}}||_F^2\right]}
\end{equation}
versus increasing signal to measurement noise ratio (SMNR), $\text{SMNR}=E{\left[ ||\mathbf{X}||_F^2\right]}/ E\left[ ||\mathbf{e}||_F^2\right]$ where $\mathbf{e} \sim \mathcal{N}(\mathbf{0}, \sigma^2 \mathbf{I})$ is the noise. The results are also compared to the Cram\'{e}r-Rao bound (CRB) defined for unbiased low-rank matrix estimators
\begin{equation}
\text{CRB}(\mathbf{X}) \leq E_{y|X}{\left[ ||\mathbf{X}-\hat{\mathbf{X}}(\mathbf{y})||_F^2 \right]}.
\end{equation}
The expressions of CRB for unstructured or structured low rank matrices are derived in \cite{Tang-LowerBounds, Werner-ReducedRank,Zachariah-ALSforLR}.

\begin{figure*}[t]
   \centering
   \begin{minipage}[t]{0.49\linewidth}
   \centering
  \includegraphics[width=8cm]{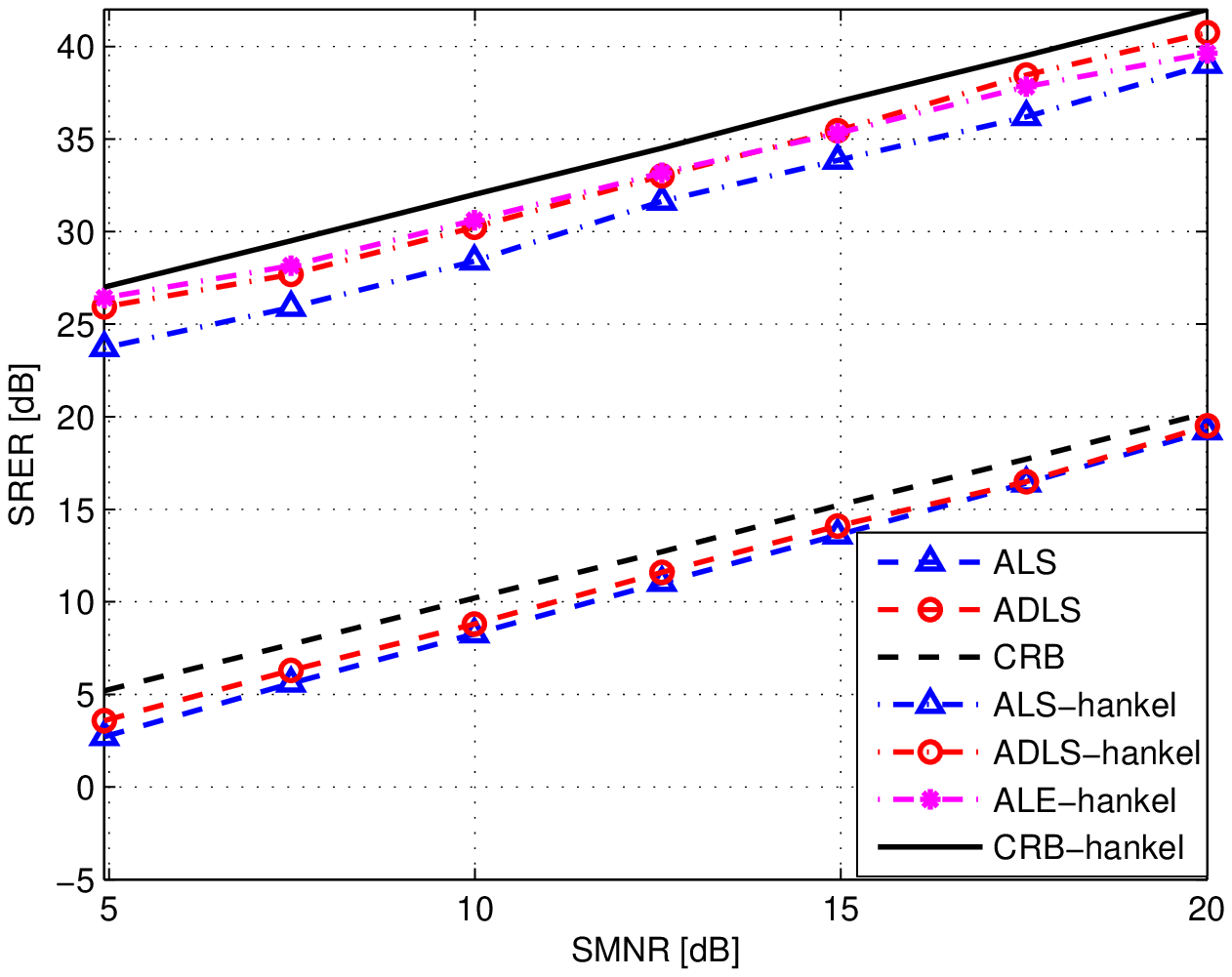} 
\centerline{(a)} 
   \end{minipage}
   \begin{minipage}[t]{0.49\linewidth}
   \centering
   \includegraphics[width=8cm]{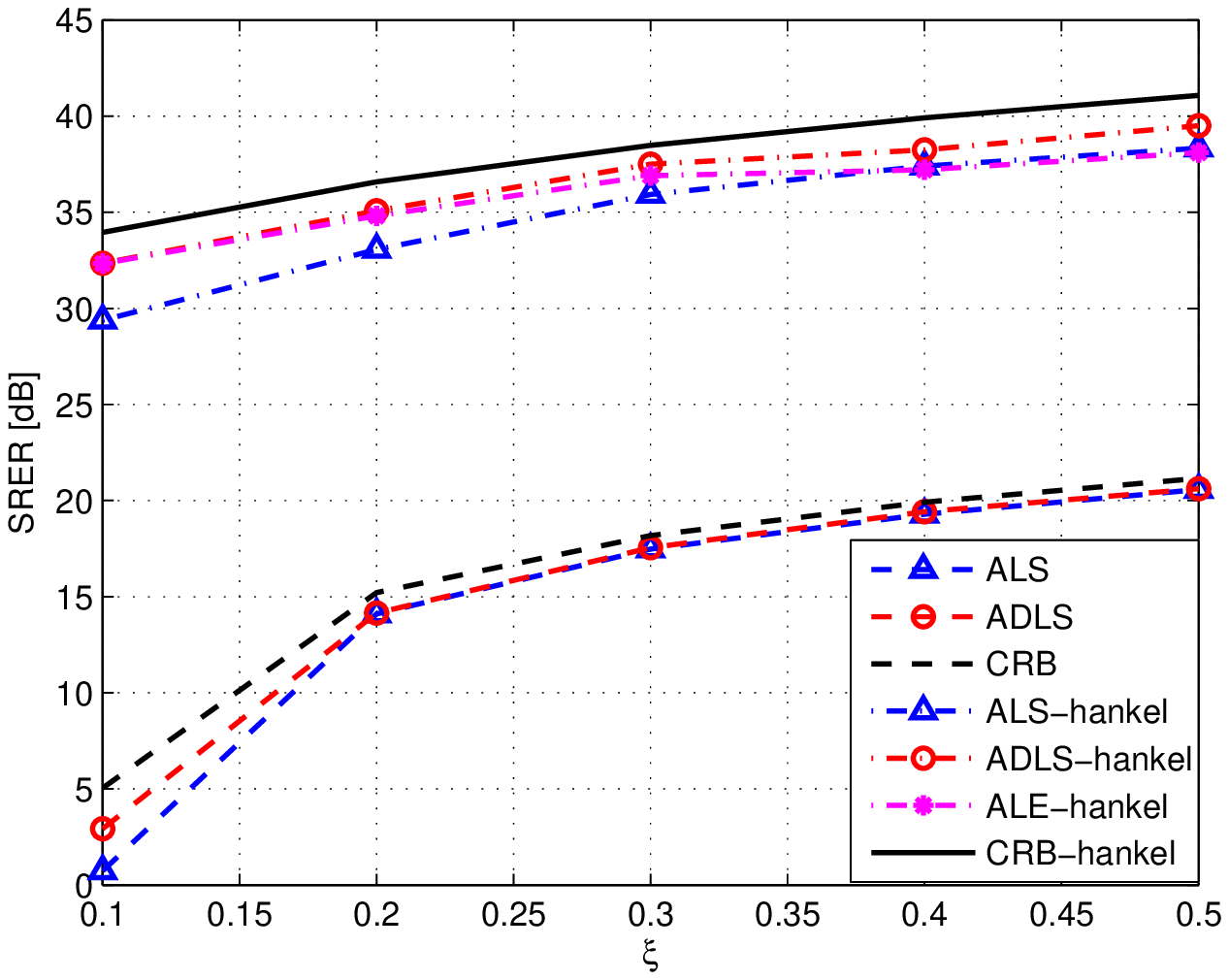} 
\centerline{(b)} 
   \end{minipage}
  \caption{Reconstruction comparison of ALS, ALE, ADLS, and CRB for unstructured and Hankel matrices, $n_1=n_2=80, \xi=0.2, r=4$ (a) SMNR versus SRER for $\xi=0.2$; (b) $\xi$ versus SRER for $\text{SMNR}=15\ \text{dB}$.}\label{fig:D-ALS_vs_ALS1}
\end{figure*}

\subsection{Simulation results}\label{sec:Simulations}
The reconstruction performances of low-rank matrices $\mathbf{X}$ with and without the prior knowledge of the Hankel structure are compared in Fig. \ref{fig:D-ALS_vs_ALS1}. Without the prior knowledge, the performance of ADLS is drawn along with the curves of the ALS and the Cram\'{e}r-Rao Bounds. While for Hankel $\mathbf{X}$, ALE curve is also included for comparison. Each data point is obtained by repeating $500$ Monte Carlo simulations, and for each iteration a new realization of $\mathbf{X}$, sensing function $\mathcal{A}$ (or $\mathbf{A}$) and $\mathbf{e}$ are generated correspondingly.

In Fig. \ref{fig:D-ALS_vs_ALS1} (a) we increase the SMNR with a fixed sampling fraction and observe their SRER. The sampling fraction $\xi = m/(n_1n_2)=0.2$, $\mathbf{X}$ has size $80 \times 80$ and rank $r=4$. The reconstruction gain is increased significantly when the matrix structure is taken into consideration. In detail, 
the performances of both ALE and ADLS are shown to be $2.5$ and $0.5$ dB better than the results of ALS with respect to Hankel structured and unstructured matrices, respectively. In particular, the curve of ADLS is close to the ALE curve when the SMNR is low, but reveals a better performance in terms of SRER when the SMNR comes more than $15$dB. In Fig. \ref{fig:D-ALS_vs_ALS1} (b) we fix $\text{SMNR}=15$ dB and let the sampling rate increase from $0.1$ to $0.5$, while keeping $n_1=n_2=80, r=4$ the same as (a). Fig. \ref{fig:D-ALS_vs_ALS1} (b) also indicates the superiority of ADLS/ALE over ALS. In addition, ADLS has a better recovery than ALE especially when $\xi>0.2$.

\section{Conclusion}
In this paper two novel algorithms are developed to strengthened the classical ALS algorithm by better using the alternating techniques for reconstruction of low-rank matrices. 
The algorithms are capable of recovering low-rank matrices in the general underdetermined setup as well as for structured matrices such as Hankel matrices. Simulations indicated that the proposed algorithms achieve better performances that are closer to the Cram\'{e}r-Rao bound.


\appendix

\begin{prop}
The updates of $\mathbf{R}$ and $\mathbf{Z}$ can be realized as (\ref{eq:update_R}) and (\ref{eq:update_Z}), respectively.
\end{prop}
\begin{proof}
Because $f$ and $g$ are differentiable, the \emph{primal feasibility} of $\mathbf{R}-\mathbf{Z}=0$ can be substituted with the \emph{dual feasibility}: $\bigtriangledown_{\mathbf{R}} L_{\lambda}(\mathbf{R},\mathbf{Z},\mathbf{U})=0$ and $\bigtriangledown_{\mathbf{Z}} L_{\lambda}(\mathbf{R},\mathbf{Z},\mathbf{U})=0$. Specifically,
\begin{equation}
\begin{split}
0&= \bigtriangledown_{\mathbf{R}}\left\{f(\mathbf{R})+\lambda||\mathbf{R}-\mathbf{Z}+\mathbf{U}||_2^2\right\} \\
&=2\mathbf{P}^T\mathbf{P}\text{vec}(\mathbf{R})-2\mathbf{P}^T\bar{\mathbf{y}}+2\lambda\text{vec}(\mathbf{R-\mathbf{Z}+\mathbf{U}}).
\end{split}
\end{equation}
Move terms containing $\mathbf{R}$ to one side of the equation,
\begin{equation}
\left(\mathbf{P}^T\mathbf{P}+\lambda \mathbf{I}\right)\text{vec}(\mathbf{R})= \mathbf{P}^T \bar{\mathbf{y}} +\lambda\text{vec}(\mathbf{Z}-\mathbf{U}),
\end{equation}
from which (\ref{eq:update_R}) can be derived straightforwardly. For the $\mathbf{Z}\text{-min}$ step,
\begin{equation}
\begin{split}
0&= \bigtriangledown_{\mathbf{Z}}\left\{g(\mathbf{Z})+\lambda||\mathbf{R}-\mathbf{Z}+\mathbf{U}||_2^2\right\} \\
&=\mu 2\mathbf{S}^T \left( \mathbf{S}\mathbf{Z}-\hat{\mathbf{X}} \right)-2\lambda(\mathbf{R-\mathbf{Z}+\mathbf{U}}).
\end{split}
\end{equation}
Likewise, move terms containing $\mathbf{Z}$ to one side of the equation,
\begin{equation}
\left(\mu \mathbf{S}^T\mathbf{S}+\lambda \mathbf{I}\right)\mathbf{Z}= \mu \mathbf{S}^T \hat{\mathbf{X}} +\lambda(\mathbf{R}+\mathbf{U}),
\end{equation}
from which (\ref{eq:update_Z}) can be derived.
\end{proof}
 The expressions of $\mathbf{L}$ and $\mathbf{S}$ in (\ref{eq:update_L}) and (\ref{eq:update_S}) can be obtained in a similar way.

\end{document}